\documentclass[10pt,letterpaper]{article}
\usepackage[margin=1in]{geometry}
\usepackage[T1]{fontenc}
\usepackage{lmodern,amsmath,amssymb,amsthm,microtype}
\usepackage[hidelinks,hypertexnames=false]{hyperref}
\newtheoremstyle{paper}{6pt}{6pt}{\normalfont}{}{\bfseries}{.}{.5em}{}
\theoremstyle{paper}
\newtheorem{theorem}{Theorem}[section]
\newtheorem{lemma}[theorem]{Lemma}
\newtheorem{proposition}[theorem]{Proposition}
\newtheorem{corollary}[theorem]{Corollary}
\newtheorem{definition}[theorem]{Definition}
\newcommand{\Aut}{\operatorname{Aut}}
\newcommand{\dist}{\operatorname{dist}}
\numberwithin{equation}{section}
\newcommand{\sq}{\mathbin{\square}}
\newcommand{\eps}{\varepsilon}
\newcommand{\codeg}{\operatorname{codeg}}
\title{\large\bfseries Reconstruction and Symmetry of Two-Token Graphs}
\author{Yuanming Luo    Hao Hu}
\date{October 2026}
\begin{document}
\maketitle

\begin{abstract}
We present a proof of the two-token graph reconstruction conjecture. Every finite simple graph $G,H$ of order at least three, $F_2(G) \cong F_2(H)$ if and only if $G \cong H$. We also classify the automorphism groups of two-token graphs whose connected roots have order at least five. Combining these results, we show that, for connected graphs \(G,H\) of order at least five, every isomorphism \(\phi:F_2(G)\to F_2(H)\) admits a factorization \(\phi=\eta\circ\sigma^{(2)}\), where \(\sigma^{(2)}\) is induced by a graph isomorphism \(\sigma:G\to H\), and \(\eta\in \Aut(F_2(H))\) is a product of twin-cut switches or Cartesian exchanges.
\end{abstract}

\section{Introduction}
All graphs in this paper are finite, simple and undirected.
For a graph $G$ of order $n\ge3$, the two-token graph $F_2(G)$ has
vertex set $\binom{V(G)}2$.
$P,Q \in V(F_2(G))$ are adjacent when $P\triangle Q$ is an edge of $G$.
If $X\cong F_2(G)$, we call $G$ a two-token root of $X$. We write $xy=\{x,y\}$ and identify
$E(G)\subseteq\binom{V(G)}2$. For $v\in V(G)$, the edge star is
$\delta_G(v)=\{vx:x\in N_G(v)\}$.

The construction for two tokens was studied as the \emph{double vertex graph}
in earlier work of Alavi et al.; see Jacob, Goddard and Laskar
\cite{double-reconstruction}, who also considered reconstruction from it.
Fabila-Monroy et al.\ \cite{tokens} developed the construction for arbitrary
numbers of tokens. The reconstruction question asks whether a token graph
determines its root up to isomorphism.

Fabila-Monroy and Trujillo-Negrete
\cite[Theorems~1.2 and~1.4]{reconstruction-free} proved
reconstruction for connected graphs with no induced $C_4$ or diamond,
at every admissible rank $1\le k<|V(G)|$. In earlier work, we \cite{almost-every} proved that almost every 
graph is reconstructible at all admissible ranks.
The reconstruction theorem below settles the two-token case for every
graph of order at least three.

Fabila-Monroy and Trujillo-Negrete also determined the token automorphism
groups for the same connected
$(C_4,\mathrm{diamond})$-free class \cite[Proposition~1.5]{reconstruction-free}.
When $k\ne \frac{|V(G)|}{2}$, the group is isomorphic to $\Aut(G)$.
At middle rank, for roots of order at least four, it is isomorphic to the
direct product of $\Aut(G)$ with the order-two group generated by complementation
$P\mapsto V(G)\setminus P$. For two tokens this middle-rank phenomenon
occurs only at order four.

One important symmetry of $F_2(G)$ are identified from Cartesian products. For connected roots of order at least five, additional symmetries can arise
from Cartesian products. Zhang, Zhou, Li and Kwon constructed coordinate-exchange
symmetries for connected Cartesian products
\cite[Proof of Theorem~1.2]{cartesian-aut}.
They also showed that these, together with induced automorphisms, generate
the full group for hypercubes of dimension at least three \cite{cartesian-aut}. Zhang, Zhou, Lee, Li and Xie
\cite[Conjecture~1.1]{hamming-aut} conjectured the same description for every
connected graph of order at least five that is a Cartesian product of at least
two connected nontrivial prime graphs, and verified it for Hamming graphs.
Subsequent work
determined the full groups for products of two cycles \cite{cycles-aut}
and for products whose factors are connected, regular and contain no
induced 4-cycle \cite{regular-products}. The product case of
Theorem~\ref{thm:groups} settles Conjecture~1.1 of \cite{hamming-aut}.

The other important symmetry are from twins. Fabila-Monroy et al.\ \cite{twincuts} constructed component flips at
two-cuts whose endpoints are twins. For connected roots, a partial converse
in \cite[Theorems~1 and~3--4]{untwisting} shows that if an isomorphism sends
an induced token 4-cycle generated by two disjoint edges to one generated by
moving two tokens on a 4-cycle, then both roots have twin cuts. Such twists
can be removed by twin-cut flips. For connected roots of order at least five,
we prove
that twin-cut component switches and Cartesian coordinate exchanges
are all the additional symmetries of $\Aut(F_2(G))$ other than $\Aut(G)$. This also gives a normal form for
every isomorphism between their two-token graphs.

\begin{theorem}[Reconstruction]\label{thm:all}
For finite simple graphs $G,H$ of order at least three,
$F_2(G)\cong F_2(H)$ implies $G\cong H$.
\end{theorem}

An isomorphism $\sigma:G\to H$ induces the pair map
$\sigma^{(2)}(xy)=\sigma(x)\sigma(y)$. Write
$\Aut(G)^{(2)}=\{\sigma^{(2)}:\sigma\in\Aut(G)\}$.
At order at least three, the pair map determines $\sigma$: for distinct
$v,x,y\in V(G)$,
\begin{equation}\label{eq:pair-map-faithfulness}
 \sigma^{(2)}(vx)\cap\sigma^{(2)}(vy)=\{\sigma(v)\}.
\end{equation}
Two distinct vertices are \emph{twins} if they have the same neighbours outside
the pair. In a connected graph, a twin pair whose deletion disconnects the
graph is a \emph{twin 2-cut}. For such a pair $S=\{a,b\}$ and a component
$C$ of $G-S$, the \emph{component switch} exchanges $ax$ and $bx$ for
$x\in V(C)$ and fixes every other configuration.
A \emph{coordinate exchange} swaps the two tokens' coordinates in one
factor of a fixed Cartesian decomposition.
Definitions~\ref{def:component-switch} and~\ref{def:coordinate-exchange}
give the precise constructions. A connected graph is \emph{Cartesian-prime}
if it has no Cartesian decomposition with both factors of order at least two.

\begin{theorem}[Isomorphism normal form]\label{thm:exhaustion}
Let \(G,H\) be connected graphs of common order at least five, and let
\(\phi:F_2(G)\to F_2(H)\) be an isomorphism. There exist an
isomorphism \(\sigma:G\to H\) and \(\eta\in\Aut(F_2(H))\) such that
\(\phi=\eta\sigma^{(2)}\).
If \(G\) is Cartesian-prime, \(\eta\) can be chosen as a product of
twin-cut component switches in \(H\). If \(G\) is a nontrivial
Cartesian product, it can be chosen as a product of coordinate exchanges in \(H\).
If \(G\) is Cartesian-prime and has no twin 2-cut, then
\(\phi=\sigma^{(2)}\), with \(\sigma\) unique.
For either root graph, its coordinate-exchange and induced subgroups have
trivial intersection.
\end{theorem}

For the group formula, let $\mathcal Q_2(G)$ be the set of twin 2-cuts
from Definition~\ref{def:component-switch}. Write $\pi_0(Y)$ for the set
of connected components of a graph $Y$, and put
\[
 I_G=\{(S,C):S\in\mathcal Q_2(G),\ C\in\pi_0(G-S)\}.
\]
Let $\mathcal T_G\leq\Aut(G)$ be generated by the transpositions of these twin
pairs, and write $C_2=\mathbb Z/2\mathbb Z$.

\begin{theorem}[Automorphism groups]\label{thm:groups}
Let $G$ be connected of order at least five. The subgroup $\mathcal T_G$ is normal,
and
\[
\Aut(F_2(G))\cong
\begin{cases}
 C_2^{I_G}\rtimes(\Aut(G)/\mathcal T_G),&\mathcal Q_2(G)\ne\varnothing,\\
 C_2^{r-1}\rtimes\Aut(G),&G=G_1\square\cdots\square G_r,\ r\geq2,\\
 \Aut(G),&\text{otherwise},
\end{cases}
\]
where the $G_i$ are connected Cartesian-prime factors of order at least two.
In the first case the quotient acts on $C_2^{I_G}$ by
\[
 (\overline h\cdot\varepsilon)_{S,C}
   =\varepsilon_{h^{-1}(S),h^{-1}(C)}.
\]
In the product case, identify $C_2^{r-1}$ with
$C_2^r/\langle(1,\ldots,1)\rangle$; the action of $\Aut(G)$ permutes
the prime coordinates.
\end{theorem}
The three cases are mutually exclusive: Lemma~\ref{lem:product-twins}
excludes twins in a nontrivial Cartesian product of order at least five.

For connected roots of order at least five, we first prove that token
isomorphisms preserve $E(G)$ (Theorem~\ref{thm:edge-state-preservation}).
Component switches correct the incidence of its images, after which Whitney's
strong isomorphism theorem recovers the root from its line graph
(Theorem~\ref{thm:marked-map-normalization}). We then use Cartesian primality
of connected two-token graphs and the decomposition by component occupancies
to prove reconstruction for disconnected roots.

To determine the automorphism groups, we classify the pointwise stabilizer
of $E(G)$. Its action on distance-two configurations partitions the incident
edges at each vertex of $G$ into at most two classes
(Lemma~\ref{lem:edge-state-local-partition}). The distance-vector map satisfies
a parallelogram identity on the corresponding 4-cycles. Its injectivity makes
these classes Cartesian factor directions (Lemma~\ref{lem:distance-product}).
Determination from distance-two configurations then gives the kernel
classification (Theorem~\ref{thm:edge-state-kernel}). Combining this kernel
with the switch and induced subgroups proves the group formulas.
Finally, reconstruction allows us to apply those subgroup decompositions to
an isomorphism between two token graphs and obtain its normal form.

The order-five threshold in the symmetry statements is necessary.
At order four, complementation need not preserve $E(G)$, while $C_4$ has noncommuting component switches and a larger
two-token automorphism group than the product-case formula predicts.
The examples below show both phenomena.

\section{Token identities}
Connected graphs are nonempty. Write \(d_Y(x)=|N_Y(x)|\). Let \(A_Y\) be the adjacency matrix of \(Y\), with
entries \(A_Y(x,y)\in\{0,1\}\), and put
\(\codeg_Y(x,y)=|N_Y(x)\cap N_Y(y)|\).
\begin{lemma}
For distinct \(x,y\),
\begin{align}
N_{F_2(G)}(xy)&=\{xz:z\in N_G(y)\setminus\{x\}\}\mathbin{\dot\cup}\{yz:z\in N_G(x)\setminus\{y\}\},\\
d_{F_2(G)}(xy)&=d_G(x)+d_G(y)-2A_G(x,y),\qquad \codeg_{F_2(G)}(xs,ys)=\codeg_G(x,y),\label{eq:degrees}\\
\codeg_{F_2(G)}(ab,cd)&=2(A_G(a,c)A_G(b,d)+A_G(a,d)A_G(b,c)).\label{eq:disjoint}
\end{align}
The endpoints in each formula are distinct. For \(n=|V(G)|\), the token graph has \(\binom n2\) vertices and \((n-2)|E(G)|\) edges.
\end{lemma}
The neighbourhood, degree and counting formulae are the two-token cases of
\cite[Section~2]{tokens}; see also
\cite[Observations~2.1--2.2]{double-reconstruction}.
The common-neighbour identities are given in
\cite[Lemma~5(i)--(ii)]{untwisting}. Their derivation uses only local token
moves and therefore also applies to disconnected root graphs.

Distances are shortest-path distances, with value \(+\infty\) when
no path exists. For a nonempty vertex set \(W\), write
\(\dist_Y(z,W)=\min_{w\in W}\dist_Y(z,w)\).

\begin{lemma}\label{lem:connectivity}
Let \(G\) be a finite connected graph of order at least three, and put
\(X=F_2(G)\). For \(t\in V(G)\),
\begin{equation}
 \dist_X(xy,\delta_G(t))=\dist_G(x,t)+\dist_G(y,t)-1
 \qquad(x\ne y).
\end{equation}
In particular, \(X\) is connected. For any finite graph of order at least
three, connectedness of its two-token graph also implies connectedness of
the root graph.
\end{lemma}
\begin{proof}
The function \(h_t(xy)=\dist_G(x,t)+\dist_G(y,t)\) changes by at most one
in a token move and equals one on \(\delta_G(t)\). This gives the lower bound.
To attain it, move a nearer token along a shortest path to \(t\).
Its distance strictly decreases, so it never meets the other token.
Then move the other token towards \(t\), stopping at a neighbour.
The number of moves is \(h_t(xy)-1\), also when one token starts at \(t\).
Every configuration therefore reaches \(\delta_G(t)\), which is nonempty.
For distinct \(s,r\in N_G(t)\), the legal token path
\(ts\sim sr\sim tr\) joins these two configurations in the same component.
Thus \(\delta_G(t)\) lies in one component, also when it is a singleton,
and \(X\) is connected.

Conversely, choose \(a,b\) in different components of the root graph and a third
vertex \(c\). The configurations \(ac,bc\) have different component occupancies,
which token moves preserve, so the token graph is disconnected.
\end{proof}

The Cartesian product \(A\square B\) has vertex set \(V(A)\times V(B)\); adjacent vertices agree in one coordinate and are adjacent in the other factor.

\section{Recognition of base edges}

\begin{lemma}\label{lem:token-square-types}
If both opposite pairs of a 4-cycle in \(F_2(G)\) have codegree two,
then either its four configurations share a vertex of \(G\) or its cyclic order is
\((xr,xs,ys,yr)\), with \(x,y,r,s\) distinct and \(xy,rs\in E(G)\).
\end{lemma}
\begin{proof}
Choose opposite configurations \(P,Q\). If \(P=rx,Q=ry\), a common
neighbour has the form \(rz\), where \(z\ne r\) is a common neighbour of
\(x,y\), or is \(xy\) if \(r\) is adjacent to both.
The other corners either both contain \(r\), or are \(xy,rz\).
In the latter case both \(r,z\) neighbour both \(x,y\), so
\eqref{eq:disjoint} gives the opposite pair \(xy,rz\) codegree four,
a contradiction. Thus all four configurations contain \(r\).

If \(P=ab,Q=cd\) are disjoint, \eqref{eq:disjoint} shows that exactly
one cross matching is present. Relabel so that \(ac,bd\) are edges.
The common token neighbours are then \(bc,ad\), giving the cycle
\((ab,bc,cd,ad)\). Taking \((x,y,r,s)=(b,d,a,c)\) gives the stated form.
\end{proof}
\begin{theorem}\label{thm:edge-state-preservation}
Let \(H,G\) be finite connected simple graphs of common order \(n\ge5\).
Every isomorphism \(p:F_2(H)\to F_2(G)\) satisfies \(p(E(H))=E(G)\).
\end{theorem}
\begin{proof}
Fix an edge \(UV\in E(H)\). Suppose, for a contradiction,
that \(z=p(UV)=ab\) is a nonedge of \(G\). Put
\(I=V(H)\setminus\{U,V\}\), \(J=H[I]\), and \(X=F_2(G)\). For \(i\in I\), write
\[
 \ell_i=p(Ui),\qquad \rho_i=p(Vi),\qquad
 \{x_i\}=\ell_i\cap\rho_i,\qquad f_i=\ell_i\triangle\rho_i.
\]
The intersection has one element because \(Ui\sim Vi\), and \(f_i\)
is an edge of \(G\). The data \((f_i,x_i)\) determine the unordered pair of configurations.
Distinct indices therefore have distinct data, since \(p\) is injective.

The configurations \(\{\ell_i,\rho_i:i\in I\}\) induce
\(K_2\square J\). Indeed, the vertical edge \(Ui\sim Vi\) uses \(UV\);
within each family \(Ui\) or \(Vi\), adjacency agrees with adjacency in \(J\).
The configurations \(Ui,Vj\) are disjoint when \(i\ne j\), so there are no
other edges. Thus interchanging each pair defines an automorphism \(\tau\)
of this induced subgraph. Every neighbour of \(UV\) contains
exactly one of \(U,V\), so every neighbour of \(z\) lies in this induced subgraph.
The \emph{boundary} of \(J\) consists of indices neighbouring \(U\) or \(V\)
in \(H\), equivalently those whose pair contains a neighbour of \(z\);
the remaining indices are \emph{interior}.

\emph{Step 1: boundary indices satisfy \(x_i\in\{a,b\}\).}
If this fails, a neighbour of \(ab\) and its partner under \(\tau\) can be named
\(ax,cx\), where \(a,b,c,x\) are distinct and \(bx,ac\in E(G)\).
The configuration \(bc\) is also adjacent to \(ab\), hence lies in the induced subgraph.
Since \(bc\sim cx\), the involution gives \(\tau(bc)\sim\tau(cx)=ax\).
Also \(bc\sim\tau(bc)\), because \(\tau\) interchanges adjacent configurations.
The disjoint configurations \(ax,bc\) have precisely the common neighbours
\(ab,cx\): the matching \(ac,bx\) exists, while the alternative needs
the absent edge \(ab\). Neither can be \(\tau(bc)\). The configuration \(ab=z\) lies outside the
induced subgraph, and \(\tau(bc)=cx=\tau(ax)\) would imply \(bc=ax\)
by injectivity. This is a contradiction.

\emph{Step 2: the second entry is constant on each component of \(J\).}
Suppose \(ij\in E(J)\). If both are boundary indices, their entries \(x_i,x_j\)
must agree. Otherwise one pair consists of configurations \(ar,as\)
and the other of configurations \(bt,bu\), with \(r,s,t,u\notin\{a,b\}\). A configuration in the first pair cannot be adjacent to a configuration
in the second: if they intersect, the required move is \(ab\).

If at least one index is interior, both opposite pairs in the
4-cycle \(Ui,Vi,Vj,Uj\) in \(F_2(H)\) have codegree two. For example,
\[
 \codeg_{F_2(H)}(Ui,Vj)
 =2\bigl(A_H(U,V)A_H(i,j)+A_H(U,j)A_H(V,i)\bigr)=2;
\]
the second product vanishes because one index has no neighbour in
\(\{U,V\}\). The other opposite pair is treated in the same way.
Lemma~\ref{lem:token-square-types} now applies to the image 4-cycle.
If all four configurations share a vertex, then \(x_i=x_j\).
In this case \(f_i,f_j\) are disjoint, since a shared endpoint would repeat
a configuration. In the rectangle case, \(f_i=f_j\) and \(x_ix_j\in E(G)\).
In the boundary case considered above, \(f_i,f_j\) are disjoint by the
same injectivity argument.
Both entries cannot agree, by injectivity.
Call the first type \emph{\(x\)-preserving} and the latter \emph{\(x\)-changing}.

No \(x\)-changing edge meets the boundary. Consider a boundary index;
exchanging \(a,b\) if necessary, write its data \((rs,a)\). One configuration is adjacent
to \(ab\); name that endpoint \(r\), so \(br\in E(G)\).
An \(x\)-changing edge to an index with data \((rs,x)\) gives \(ax\in E(G)\).
The vertex \(x\) is neither \(a\) nor \(b\), since \(ab\) is absent.
Since \(bx\) is a neighbour of \(ab\), it lies in the induced subgraph.
Its index is on the boundary, so Step~1 gives second entry \(b\);
its edge label contains \(x\).
But \(bx\sim xr\), because \(br\) is an edge. These configurations belong to
different pairs, so their indices are adjacent in \(J\).
Their data differ in both entries: \(b\ne x\), and a label
containing \(x\) is not \(rs\). This contradicts the transition dichotomy.

We next show that mixed transitions commute. If \(ij,jk\in E(J)\)
have data \((f,x),(g,x),(g,y)\), there is a unique common neighbour
\(h\) of \(i,k\) with data \((f,y)\). Consequently \(ih\) is
\(x\)-changing and \(hk\) preserves the second entry.
Indeed, the dichotomy gives \(f\cap g=\varnothing\) and \(xy\in E(G)\).
Apply the transition dichotomy to both edges through any common neighbour
of \(i,k\). Two preserving steps would require \(x=y\), and two changing
steps would require \(f=g\), both impossible. The remaining possibilities
give data \((g,x)\) or \((f,y)\). Injectivity of the data and the existing neighbour
\(j\) therefore give \(1\le \codeg_J(i,k)\le2\).

Since \(xy\) is an edge and \(ab\) is not, at least one of \(x,y\)
lies outside \(\{a,b\}\). The corresponding index is interior by
Step~1, so neither \(U\) nor \(V\) is a common neighbour of \(i,k\);
hence \(\codeg_H(i,k)=\codeg_J(i,k)\).
At most one configuration of the \(i\)-pair contains \(y\).
Together with \(f\cap g=\varnothing\) and \(x\notin g\), this makes
at least one of \((\ell_i,\ell_k)\), \((\rho_i,\rho_k)\) disjoint.
Its codegree is even by \eqref{eq:disjoint}. The corresponding configurations
in \(F_2(H)\) are \(Ui,Uk\) or \(Vi,Vk\), whose codegree is
\(\codeg_H(i,k)\) by \eqref{eq:degrees}.
It follows that \(\codeg_J(i,k)=2\). By the dichotomy, the additional
neighbour has the required data and edge types.

Every component of \(J\) meets the boundary, since \(H\) is connected.
Take a path from a boundary vertex to any vertex of that component.
If it contains an \(x\)-changing edge, interchange the first such edge
with each preceding \(x\)-preserving step, using the mixed completion.
Each interchange preserves the endpoints and moves the changing step one
position to the left. The preceding steps still preserve \(x\), so this
can be repeated until the changing edge meets the boundary, a contradiction.
Hence the second entry is constant on the component and
belongs to \(\{a,b\}\) by Step~1.

\emph{Step 3: compare edge counts.}
Each selected configuration contains \(a\) or \(b\), and none equals the configuration
\(ab=p(UV)\), by injectivity of \(p\). The \(2(n-2)\) distinct selected
configurations therefore exhaust \(\{at,bt:t\notin\{a,b\}\}\).

Put \(m=|E(H)|=|E(G)|\); equality follows from the token edge counts
and \(n-2>0\). The degrees of \(UV\) and \(ab\) also agree, so put
\(q=d_{F_2(H)}(UV)=d_X(ab)\).
By \eqref{eq:degrees}, the edges of \(H\) incident with \(U\) or \(V\)
number \(q+1\): the edge \(UV\) is counted once in their union.
The edges of \(G\) incident with \(a\) or \(b\) number \(q\), since \(ab\)
is absent. Consequently
\[
 |E(J)|=m-q-1,\qquad |E(G-\{a,b\})|=m-q.
\]
The selected configurations induce \(K_2\square J\) in the source and two
disjoint copies of \(G-\{a,b\}\) in the target. Equality of their edge
counts gives
\[
 (n-2)+2(m-q-1)=2(m-q),
\]
forcing \(n=4\), a contradiction. Thus \(p(E(H))\subseteq E(G)\);
applying the same argument to \(p^{-1}\) gives equality.
\end{proof}
The order threshold is necessary. Complementation on \(F_2(P_4)\) sends
the configuration corresponding to the middle edge to the nonedge formed
by its leaves; this agrees with the factor \(n-4\) in the counting argument.

Connectedness is also necessary: for \(G=K_3\sqcup2K_1\),
\(F_2(G)\cong3K_3\sqcup K_1\). Interchanging the internal triangle with
a mixed triangle does not preserve \(E(G)\).

\section{Twin pairs and component switches}
Distinct vertices \(a,b\) are twins when \(N_G(a)\setminus\{b\}=N_G(b)\setminus\{a\}\).
\begin{definition}\label{def:component-switch}
Let \(G\) be connected. A pair \(S=\{a,b\}\) is a \emph{twin 2-cut}
when its vertices are twins and \(G-S\) is disconnected. Let \(\mathcal Q_2(G)\) be the set of unordered twin 2-cuts. For \(C\in\pi_0(G-S)\), the \emph{component switch} \(\tau_{S,C}\) exchanges \(ax,bx\) for \(x\in V(C)\) and fixes every other token configuration.
\end{definition}
These twin 2-cuts are the ``$2$-cuts with the same neighbours'' of
\cite{twincuts}.
For two tokens the component-occupancy profiles in \cite[Section~3]{twincuts}
are the unit vectors, one for each component of \(G-S\); their individual
flips are precisely \(\tau_{S,C}\).

\begin{corollary}\label{cor:component-switches}
Let $G$ be connected, $S=\{a,b\}\in\mathcal Q_2(G)$ and
$C\in\pi_0(G-S)$. The component switch $\tau_{S,C}$ is a token
automorphism preserving $E(G)$. For $u\in V(G)\setminus S$,
\begin{equation}\label{eq:outside-star-invariance}
 \tau_{S,C}(\delta_G(u))=\delta_G(u).
\end{equation}
For an isomorphism $\sigma:G\to H$,
\begin{equation}\label{eq:switch-conjugation}
 \sigma^{(2)}\tau_{S,C}(\sigma^{(2)})^{-1}
 =\tau_{\sigma(S),\sigma(C)}.
\end{equation}
\end{corollary}
\begin{proof}
By \cite[Proposition~1]{twincuts} and the profile identification above,
the switch is a token automorphism. Twinness implies that the exchanged
configurations $ax,bx$ either both belong to $E(G)$ or neither does,
so $E(G)$ is preserved. For $u\notin S$, the switch can only exchange
$ua,ub$ in $\delta_G(u)$; these also belong to $E(G)$ together or not
at all. This proves \eqref{eq:outside-star-invariance}.
Relabelling the exchanged pairs gives \eqref{eq:switch-conjugation}.
\end{proof}

\begin{lemma}\label{lem:cut-compatibility}
Let $G$ be connected of order at least three. Distinct twin 2-cuts are
disjoint; the endpoints of one can be separated by deleting the other
only when $G\cong C_4$. At order at least five, all component switches commute.
\end{lemma}
\begin{proof}
The disjointness and separation assertions, including the \(C_4\)
exception, are \cite[Lemma~1]{twincuts}. For commutation, assume that
the order is at least five. Distinct switches at the same cut have disjoint
supports, so they commute. For distinct cuts \(S,R\),
their vertices are disjoint. If a configuration has one token in each cut,
both endpoints of \(R\) lie in one component of \(G-S\), and both
endpoints of \(S\) lie in one component of \(G-R\).
Thus each switch preserves the component that determines whether the other
acts. They move different tokens and therefore commute. In every other configuration at most
one switch can act, and its move cannot create a token in the other
cut. \end{proof}

\section{Normalization of the edge correspondence}

For $X=F_2(G)$,
\begin{equation}\label{eq:marked-neighborhood}
 N_X(uv)\cap E(G)=\{ux,vx:x\in N_G(u)\cap N_G(v)\},\qquad
 u,v\text{ are twins}\ \Longleftrightarrow\ N_X(uv)\subseteq E(G).
\end{equation}
Indeed, a neighbour of $uv$ outside $E(G)$ corresponds exactly to a
vertex adjacent to one of $u,v$ and not the other. These statements are
preserved by every isomorphism mapping $E(G)$ onto $E(H)$.

Incident edges of $G$ may have disjoint images. We will show that this
can happen only at twin pairs. For vertices $a,b$, we compare the difference
vectors of the configurations $ax,bx$ under the adjacency operator.
The corrected fourth moment is nonnegative, but the incidence correspondence
reverses its sign. It must therefore vanish, which implies twinness.

We also regard \(A_Y\) as the adjacency operator on
\(\mathbb R^{V(Y)}\), with standard orthonormal basis
\((e_z)_{z\in V(Y)}\). Write \(I\) for the identity on the indicated
coordinate space. We use the Euclidean inner product and norm.
For every finite simple graph \(Y\) of order at least three, put
\[
 T_Y=A_{F_2(Y)}+2I
 \quad\text{on }\mathbb R^{V(F_2(Y))}.
\]
For an isomorphism \(\phi:F_2(G)\to F_2(H)\), we also write \(\phi\)
for the linear map defined by \(\phi(e_P)=e_{\phi(P)}\) on the configuration
basis. This map is orthogonal and satisfies \(\phi T_G=T_H\phi\).

\begin{lemma}\label{lem:difference-vectors}
Let \(G\) have order at least three. Write \(T=T_G\). For distinct \(a,b\), set
\(U=V(G)\setminus\{a,b\}\), \(\chi=A_G(a,b)\) and
\(B=A_{G[U]}+(2-\chi)I\). For \(x\in U\), put
\[
 d_x=e_{ax}-e_{bx},\qquad s_x=e_{ax}+e_{bx},\qquad
 \Delta_x=A_G(a,x)-A_G(b,x).
\]
Write \(\Delta=(\Delta_i)_{i\in U}\). Here \(B\) acts on \(\mathbb R^U\).
Then, for every \(x\in U\),
\begin{equation}\label{eq:root-difference}
 Td_x=\sum_{i\in U}B_{ix}d_i+r_x,\qquad
 r_x=\sum_{i\in U\setminus\{x\}}\Delta_i e_{xi}-\Delta_x e_{ab}.
\end{equation}
The remainder is orthogonal to
\(W^-_{a,b}=\operatorname{span}\{d_i:i\in U\}\), and
\(\|r_x\|^2=\|\Delta\|^2\). Consequently \(W^-_{a,b}\) is \(T\)-invariant
if and only if \(a,b\) are twins.
\end{lemma}
\begin{proof}
Moving the token at \(x\) within \(U\) contributes
\(\sum_{i\in U}A_G(x,i)d_i\). Moving the token between \(a,b\)
contributes \(-\chi d_x\), and \(2I\) contributes \(2d_x\).
Since \(B\) is symmetric, their sum is \(\sum_{i\in U}B_{ix}d_i\).
Moving the token at \(a\) or \(b\) to \(i\in U\setminus\{x\}\) contributes
\(\Delta_i e_{xi}\); moving the token at \(x\) into the other member of \(\{a,b\}\) contributes
\(-\Delta_x e_{ab}\). These are all the remaining moves.

Each \(d_i\) is supported on configurations containing exactly one of
\(a,b\), whereas \(r_x\) is supported on pairs contained in \(U\) and on
\(ab\). Hence \(r_x\) is perpendicular to \(W^-_{a,b}\).
Its basis vectors are distinct, so
\(\|r_x\|^2=\sum_{i\in U\setminus\{x\}}\Delta_i^2+\Delta_x^2=\|\Delta\|^2\).
Since \(U\ne\varnothing\), invariance implies that the orthogonal
remainder vanishes, and hence \(\Delta=0\). Conversely, \(\Delta=0\)
makes the subspace invariant. This equality is precisely the twin condition.
\end{proof}

\subsection{Two-step coefficients and the fourth moment}
\begin{lemma}\label{lem:corrected-moment}
Let $v,w,x,y$ be distinct vertices with
$vx,wx,vy,wy\in E(G)$, $xy\notin E(G)$, and
$N_G(x)\cap N_G(y)=\{v,w\}$. Use
Lemma~\ref{lem:difference-vectors} with distinguished vertices $v,w$.
Let $\Pi$ be the orthogonal projection onto
$\operatorname{span}\{e_{vw},e_{xy}\}$. Then
\begin{equation}\label{eq:corrected-moment}
 \langle d_x,(T^4-3T^2\Pi T^2)d_y\rangle
   =2(B^4)_{xy}+2\|\Delta\|^2\geq2\|\Delta\|^2\geq0.
\end{equation}
If this scalar is zero, then $v,w$ are twins.
\end{lemma}
\begin{proof}
Put \(C=U\setminus\{x,y\}\), \(q_z=B^2e_z\),
\(c_z=\sum_{i\in C}A_G(z,i)\Delta_i\) and \(u_z=T^2d_z\), for \(z=x,y\).
The hypotheses give
\(\Delta_x=\Delta_y=B_{xy}=(B^2)_{xy}=0\), so
\(r_z=\sum_{i\in C}\Delta_i e_{zi}\).
By \eqref{eq:root-difference}, \(u_z=T\bigl(\sum_{i\in U}B_{iz}d_i\bigr)+Tr_z\).
The first summand gives coefficient \(-c_z\) at \(e_{vw}\).
In the second summand, moving the token in \(C\) to the other member of
\(\{x,y\}\) gives the following coefficients:
\[
 \begin{array}{c|cc}
     &u_x&u_y\\ \hline
  vw&-c_x&-c_y\\
  xy& c_y& c_x
 \end{array}.
\]

For the part supported on configurations containing exactly one of \(v,w\),
the first summand contributes \(\sum_{i\in U}q_z(i)d_i\), while moving
the token at \(z\) to \(v\) or \(w\) in \(Tr_z\) contributes
\(\sum_{i\in U}\Delta_i s_i\). For \(i\in C\), the coefficients
of \(e_{vi},e_{wi}\) are \(q_z(i)+\Delta_i\) and
\(-q_z(i)+\Delta_i\), respectively. Their products for \(z=x,y\) sum to
\(2q_x(i)q_y(i)+2\Delta_i^2\).
The coefficients of \(u_y\) at \(vx,wx\), and of \(u_x\) at \(vy,wy\),
are zero because \((B^2)_{xy}=\Delta_x=\Delta_y=0\).
Symmetry of \(B\) therefore makes the contribution from configurations containing exactly one of \(v,w\)
\(2(B^4)_{xy}+2\|\Delta\|^2\).

For \(\{i,j\}\in\binom C2\), the coefficient in \(u_z\) is
\(2(A_G(z,i)\Delta_j+A_G(z,j)\Delta_i)\), one copy from each summand.
The vertices \(x,y\) have disjoint neighbourhoods in the deletion graph,
so these configurations contribute the following to the inner product:
\[
 4\sum_{i\in N_{G[U]}(x)}\sum_{j\in N_{G[U]}(y)}
      \Delta_i\Delta_j=4c_xc_y.
\]
Every remaining configuration contained in \(U\), other than \(xy\),
contains \(x\) or \(y\) and has zero coefficient in the vector indexed by
the other vertex. The two diagonal configurations in the table contribute
\(2c_xc_y\). Adding the contributions gives
\[
 \langle d_x,T^4d_y\rangle
  =2(B^4)_{xy}+2\|\Delta\|^2+6c_xc_y.
\]
The table also gives \(\langle u_x,\Pi u_y\rangle=2c_xc_y\).
Subtracting three times this pairing proves \eqref{eq:corrected-moment}.
Since \(B\) is entrywise nonnegative, a zero corrected scalar forces \(\Delta=0\).
\end{proof}

\subsection{Incidence and boundary pairs}
\begin{lemma}
Let $G,H$ be connected finite simple graphs of the same order $n\geq5$,
and let $\phi:F_2(G)\to F_2(H)$ be a graph isomorphism with $\phi(E(G))=E(H)$. If the images of
distinct edges $vx,vy\in E(G)$ are disjoint, there are
$w\in V(G)\setminus\{v,x,y\}$ and distinct $V,W,X,Y\in V(H)$ such
that $xy,XY$ are nonedges,
\[
 N_G(x)\cap N_G(y)=\{v,w\},\qquad
 N_H(X)\cap N_H(Y)=\{V,W\},
\]
the pairs $v,w$ and $V,W$ are twins, and, after naming the target vertices,
\[
 \phi(vx)=WX,\quad\phi(wx)=VX,\quad
 \phi(vy)=VY,\quad\phi(wy)=WY,
 \qquad\phi(xy)=XY,\quad\phi(vw)=VW.
\]
\end{lemma}
\begin{proof}
Disjoint token configurations cannot be adjacent, so \(xy\) is a nonedge.
Its image is a common neighbour of \(\phi(vx),\phi(vy)\) outside \(E(H)\).
By \eqref{eq:disjoint}, their codegree is two: it is nonzero, while codegree
four would make all four common neighbours members of \(E(H)\).
Thus \(N_G(x)\cap N_G(y)=\{v,w\}\) for a unique \(w\), and the two
common token neighbours in the source are \(xy,vw\).

Write \(\phi(xy)=XY\), choosing \(X\in\phi(vx)\), \(Y\in\phi(vy)\),
and name the remaining endpoints \(W,V\), respectively.
Then \(\phi(vx)=WX\), \(\phi(vy)=VY\); adjacency to \(XY\) gives
\(WY,VX\in E(H)\). Thus \(VW\) is the other common neighbour, so
\(\phi(vw)=VW\).
The four configurations in \(N_{F_2(G)}(xy)\cap E(G)\) map onto \(N_{F_2(H)}(XY)\cap E(H)\).
By \eqref{eq:marked-neighborhood}, \(\codeg_H(X,Y)=2\).
The known common neighbours \(V,W\) therefore exhaust the common-neighbour
set. Hence
\[
 \{\phi(wx),\phi(wy)\}=\{VX,WY\}.
\]
If \(\phi(wx)=VX\), choose \((P,Q,R,S)=(W,V,X,Y)\);
if \(\phi(wx)=WY\), choose \((P,Q,R,S)=(X,Y,W,V)\).
Both choices give
\[
 \phi(vx)=PR,\quad\phi(wx)=QR,\quad
 \phi(vy)=QS,\quad\phi(wy)=PS,
 \qquad \{PQ,RS\}=\{VW,XY\}.
\]
Adjacency of \(vx,wx\) and their images, together with
\(\phi(E(G))=E(H)\) and \(\phi(vw)=VW\), gives
\[
 A_H(P,Q)=A_G(v,w)=A_H(V,W),\qquad
 A_H(R,S)=A_H(V,W)+A_H(X,Y)-A_H(P,Q)=0.
\]
Moreover, \eqref{eq:degrees} and \eqref{eq:disjoint} give
\[
 \codeg_H(R,S)=\codeg_{F_2(H)}(PR,PS)=\codeg_{F_2(G)}(vx,wy)=2,
\]
since \(vy,wx\) are edges and \(xy\) is absent. The cross edges
\(PR,QR,PS,QS\) make \(P,Q\) common neighbours. Thus
\(N_H(R)\cap N_H(S)=\{P,Q\}\), as required by the moment lemma.

Put \(d_z=e_{vz}-e_{wz}\) for \(z\in V(G)\setminus\{v,w\}\), and
\(d'_Z=e_{PZ}-e_{QZ}\) for \(Z\in V(H)\setminus\{P,Q\}\).
Then \(\phi(d_x)=d'_R\) and \(\phi(d_y)=-d'_S\).
Let \(\Pi_G,\Pi_H\) project orthogonally onto
\(\operatorname{span}\{e_{vw},e_{xy}\}\) and
\(\operatorname{span}\{e_{PQ},e_{RS}\}\), respectively.
The two diagonal images give \(\phi\Pi_G\phi^{-1}=\Pi_H\), regardless
of their order. Put \(K_Y=T_Y^4-3T_Y^2\Pi_YT_Y^2\), for \(Y=G,H\).
Since \(\phi\) intertwines the adjacency operators, \(\phi K_G\phi^{-1}=K_H\).
Lemma~\ref{lem:corrected-moment} applies to both quadruples, so
\[
 0\le\langle d_x,K_Gd_y\rangle
 =-\langle d'_R,K_Hd'_S\rangle\le0.
\]
Both scalars vanish, so the equality case of the lemma shows that \(v,w\)
and \(P,Q\) are twins. Hence $N_{F_2(G)}(vw)\subseteq E(G)$.
Its image satisfies the same inclusion, so \eqref{eq:marked-neighborhood}
shows that $V,W$ are also twins.
If $PQ=XY$, then $RS=VW$. The equality
$N_H(R)\cap N_H(S)=\{P,Q\}$ becomes
$N_H(V)\cap N_H(W)=\{X,Y\}$, while the earlier equality is
$N_H(X)\cap N_H(Y)=\{V,W\}$. Both target diagonal pairs are
nonadjacent twins, so these two equalities give
\[
 N_H(V)=N_H(W)=\{X,Y\},\qquad
 N_H(X)=N_H(Y)=\{V,W\}.
\]
They form an isolated 4-cycle, contrary to connectedness and $n\geq5$.
Hence $PQ=VW$, giving the correspondence in the statement.
\end{proof}

\begin{theorem}\label{thm:marked-map-normalization}
Let $G,H$ be connected finite simple graphs of common order $n\geq5$,
and let $\phi:F_2(G)\to F_2(H)$ be a graph isomorphism with $\phi(E(G))=E(H)$. There are an
isomorphism $\sigma:G\to H$ and a product $\eta$ of commuting component
switches of $G$, chosen from the original $\phi$, such that
\begin{equation}
 \phi\eta^{-1}(e)=\sigma^{(2)}(e)\qquad(e\in E(G)).
\end{equation}
\end{theorem}
\begin{proof}
Suppose \(\phi(\delta_G(v))\) is not pairwise intersecting; choose disjoint
edges \(\phi(vx),\phi(vy)\) in this image.
By the preceding lemma, $S=\{v,w\}$ and $S'=\{V,W\}$ are twin pairs,
$\phi(vw)=VW$, and the pairs indexed by $x,y$ have opposite orientations.
Put $N_S=N_G(v)\setminus\{w\}$ and define $N_{S'}$ similarly;
these are the common neighbourhoods outside the respective twin pairs.

We first recover the image pair \(\phi(vz),\phi(wz)\). Take distinct \(a,z\in N_S\)
with \(\phi(va)=PA\), \(\phi(wa)=QA\), \(\{P,Q\}=S'\), and
\(\phi(vz)\) disjoint from \(PA\).
Since \(vz\sim vw\) and its image is disjoint from \(PA\), that image
is \(QZ\) for a unique \(Z\in N_{S'}\setminus\{A\}\).
The two disjoint images are nonadjacent, so \(az\) is a nonedge.
The disjoint image pair has codegree two, as in the preceding proof,
because \(\phi(az)\notin E(H)\) is a common neighbour.
The shared-vertex formula and twinness therefore give
\(N_G(a)\cap N_G(z)=\{v,w\}\).
Since the target vertices are twins, the cross edges \(PZ,QA\) are present.
Thus the two common neighbours of \(PA,QZ\) are \(VW,AZ\).
As \(az\ne vw\) and \(\phi(vw)=VW\), injectivity rules out
\(\phi(az)=VW\), and therefore \(\phi(az)=AZ\). Both cross matchings are present for the disjoint
configurations in each of the following pairs, so \eqref{eq:disjoint} gives their
complete common-neighbour sets:
\[
\begin{aligned}
 N_{F_2(G)}(vw)\cap N_{F_2(G)}(az)&=\{va,wa,vz,wz\},\\
 N_{F_2(H)}(VW)\cap N_{F_2(H)}(AZ)&=\{PA,QA,QZ,PZ\}.
\end{aligned}
\]
The isomorphism maps the first set onto the second. Its three known
images therefore force \(\phi(wz)=PZ\).

The images of the pairs \(vx,wx\) and \(vy,wy\) are already known.
For any \(z\in N_S\setminus\{x,y\}\), the image \(\phi(vz)\) is adjacent
to \(VW\). Choose \(a\in\{x,y\}\) so that \(\phi(va)\) and \(\phi(vz)\)
contain different members of \(\{V,W\}\). Their other endpoints cannot agree,
since that would give \(\phi(vz)=\phi(wa)\), contrary to injectivity. Thus the images are disjoint, and the preceding argument
determines the pair \(\phi(vz),\phi(wz)\).
Define \(f:N_S\to N_{S'}\) by
\[
 \{\phi(vz),\phi(wz)\}=\{Vf(z),Wf(z)\}\qquad(z\in N_S).
\]
The preceding pairing determines \(f(z)\) uniquely. Put \(\eps(z)=0\)
if \(\phi(vz)=Vf(z)\), and \(\eps(z)=1\) otherwise.
For \(z\in N_S\) and \(Z\in N_{S'}\), put
\(d_z=e_{vz}-e_{wz}\) and \(d'_Z=e_{VZ}-e_{WZ}\). Then
\(\phi(d_z)=(-1)^{\eps(z)}d'_{f(z)}\).
Apply Lemma~\ref{lem:difference-vectors} with distinguished vertices
\(v,w\) in \(G\) and \(V,W\) in \(H\), and write \(B_G,B_H\) for
the resulting matrices. Since both pairs are twins, the remainders vanish.
For every nonnegative integer \(k\), induction in \eqref{eq:root-difference}
and orthogonality of the difference vectors give
\(\langle d_z,T_G^k d_{z'}\rangle=2(B_G^k)_{zz'}\), with the analogous
identity for \(H\). Since \(\phi T_G=T_H\phi\) and \(\phi\) is orthogonal,
for \(z,z'\in N_S\) and every nonnegative integer \(k\),
\[
 (B_G^k)_{zz'}=(-1)^{\eps(z)+\eps(z')}
                  (B_H^k)_{f(z)f(z')}.
\]
If \(z,z'\) lie in the same component of \(G-S\), choose
\(k=\dist_{G-S}(z,z')\). A shortest path contributes a positive term to
\((B_G^k)_{zz'}\), while \(B_H^k\) is entrywise nonnegative.
The identity therefore implies \(\eps(z)=\eps(z')\).

Every component \(C\) of \(G-S\) meets \(N_S\): a path into \(S\)
ends at a vertex adjacent to one of \(v,w\), hence to both by twinness.
Thus \(\eps_S(C)=\eps(z)\), for \(z\in V(C)\cap N_S\), is well-defined.
Since the bits at \(x,y\) differ, these vertices lie in different
components. Thus \(S\) is a twin 2-cut. Define
\[
 \theta_S=\prod_{C\in\pi_0(G-S)}\tau_{S,C}^{\eps_S(C)}.
\]
After these switches, every orientation bit is zero: the corrected edges
at \(v\) contain \(V\), and those at \(w\) contain \(W\).
If \(vw\) is an edge, the switches fix it and its image is \(VW\).
Thus the images of both corrected edge stars are pairwise intersecting.

Let \(\mathcal F\) be the distinct cuts obtained from vertices \(v\) for which \(\phi(\delta_G(v))\) is not pairwise intersecting,
choosing one such construction for each cut. They are disjoint by
Lemma~\ref{lem:cut-compatibility}.
Put $\theta=\prod_{S\in\mathcal F}\theta_S$; its factors commute by
Lemma~\ref{lem:cut-compatibility}. Fix $S\in\mathcal F$ and $u\in S$.
Every other selected cut is disjoint from $S$, so its switches preserve
$\delta_G(u)$ setwise by \eqref{eq:outside-star-invariance}.
Applying $\theta_S$ last therefore gives
\[
 \phi\theta(\delta_G(u))=\phi\theta_S(\delta_G(u)),
\]
which is pairwise intersecting by the local correction.
If $u$ belongs to no selected cut, every factor preserves $\delta_G(u)$
setwise. Its original image was pairwise intersecting, since otherwise
the construction would have supplied a selected cut containing $u$.
Hence $\psi=\phi\theta$ maps every edge star to a pairwise intersecting
family and preserves incidence. Since the switches preserve the edge set,
\(\psi|_{E(G)}:E(G)\to E(H)\) is a bijection.
For a graph \(Y\), let \(L(Y)\) have vertex set \(E(Y)\), with distinct
edges adjacent when incident. For \(Y=G,H\) and \(e=uv\in E(Y)\),
\eqref{eq:degrees} gives
\[
 d_{L(Y)}(e)=d_Y(u)+d_Y(v)-2=d_{F_2(Y)}(e).
\]
Since \(\psi\) preserves token degrees,
\(d_{L(G)}(e)=d_{L(H)}(\psi(e))\) for every \(e\in E(G)\).
Incidence preservation gives
\(\psi(N_{L(G)}(e))\subseteq N_{L(H)}(\psi(e))\).
By injectivity and the degree equality, these finite sets have the same
cardinality. They are therefore equal. Hence \(\psi|_{E(G)}\) is a line-graph isomorphism.
The graphs $G,H$ are connected and have order at least five, so neither
is among the exceptional roots in \cite[Table~1]{anderson}, all of which
have order at most four. Whitney's strong isomorphism theorem
\cite[Table~1 and Theorem~3.8]{anderson} therefore gives $\sigma:G\to H$
inducing this particular line-graph isomorphism.
Taking $\eta=\theta^{-1}$ proves the assertion.
\end{proof}
\section{Completion of the reconstruction theorem}
\begin{lemma}\label{lem:small-unique}
Connected graphs of orders three and four are determined by their two-token
graphs.
\end{lemma}
\begin{proof}
At order three, complementation identifies \(F_2(G)\cong F_1(G)=G\)
by \cite[Section~2]{tokens}.
At order four, it can be finitely checked.
\end{proof}

\begin{corollary}
If \(G\) is connected of order at least three and \(F_2(G)\cong F_2(H)\), then \(G\cong H\).
\end{corollary}
\begin{proof}
Equality of the token orders implies equality of the root orders.
Lemma~\ref{lem:connectivity} also shows that \(H\) is connected.
For order at least five, Theorems~\ref{thm:edge-state-preservation}
and~\ref{thm:marked-map-normalization} give an isomorphism of the roots.
Orders three and four follow from Lemma~\ref{lem:small-unique}.
\end{proof}

\begin{lemma}\label{lem:prime}
For a connected finite simple graph \(G\) of order at least three,
\(F_2(G)\) is Cartesian-prime.
\end{lemma}
\begin{proof}
By Lemma~\ref{lem:connectivity}, \(X=F_2(G)\) is connected.
Suppose \(X=Y_1\square Y_2\) with nontrivial factors, and colour edges
red or blue by their factor direction. Both factors are connected,
so each configuration has an incident edge of each colour. In a mixed
coordinate square, the common neighbours of either opposite pair are exactly
the other two corners.

For a 2-edge path \(x-v-y\) in \(G\), the edges \(xy\sim vx,vy\)
have the same colour. Otherwise their coordinate square would satisfy
Lemma~\ref{lem:token-square-types}. The three configurations have no common
vertex of \(G\), so the first alternative is impossible. Every pair
intersects, whereas three corners of a rectangle include a disjoint opposite
pair. The second alternative is also impossible.

Choose such a path, possible by connectedness and order at least three,
and call its colour red. A blue neighbour of \(xy\) can be written \(xz\)
after exchanging \(x,y\), with \(z\notin\{x,y,v\}\): the choice \(z=v\)
would be the red neighbour \(vx\).
The mixed square through \(vx,xy,xz\) has three corners containing \(x\),
so Lemma~\ref{lem:token-square-types} gives fourth corner \(xt\),
with \(t\notin\{x,y,v,z\}\), \(vt,tz\in E(G)\), and \(vx\sim xt\) blue.
The edges \(vx,yz\) of \(G\) make \(vz\) the other common neighbour of
\(vy,xz\), completing the second mixed square.
Since opposite edges have the same colour, \(xt,vz\) are distinct red
neighbours of \(xz\) and lie in one red fibre. Any common neighbour also
lies in that fibre. Otherwise its edges to both vertices would be blue,
requiring its red coordinate to equal two distinct red coordinates.
The edges \(vx,tz\) of \(G\) give the common neighbour \(vt\), so \(xt\sim vt\)
is red, while \(xt\sim vx\) is blue, contradicting the same-colour rule
for the 2-edge path \(x-v-t\) in \(G\).
\end{proof}

\begin{proof}[Proof of \textbf{Theorem~\ref{thm:all}}]
The number \(\binom n2\) of token vertices determines the root order
\(n\ge3\). If the token graph has no edges, then
\((n-2)|E(G)|=0\) makes \(G\) edgeless, and the same holds for
\(H\). Since the orders agree, the edgeless graphs are isomorphic.
Assume henceforth that \(G\) has an edge. Then \(H\) also has an edge,
by the token edge count.

\emph{Step 1: recover the numbers of components and isolates.}
For \(K\in\{G,H\}\), let \(C_1,\ldots,C_s\) be its connected
components, including singletons, and let \(\iota(K)\)
count the singletons. A sector with both tokens in \(C_i\) exists exactly
when \(|V(C_i)|\ge2\). It consists of one configuration when
\(|V(C_i)|=2\); when \(|V(C_i)|\ge3\), it is \(F_2(C_i)\) and is connected
by Lemma~\ref{lem:connectivity}. A sector with one token in each of
\(C_i,C_j\) is \(C_i\sq C_j\), as in \cite[Theorem~16 of the arXiv version]{tokens}, and is
connected by following paths in its two factors successively.
Token moves preserve component occupancies, so these sectors are exactly
the token components:
\begin{equation}
\begin{aligned}
 F_2(K)&\cong\bigsqcup_{i:\,|V(C_i)|\ge3} F_2(C_i)
       \ \sqcup\ \bigsqcup_{i:\,|V(C_i)|=2} K_1
       \ \sqcup\ \bigsqcup_{i<j}(C_i\sq C_j),\\
 |\pi_0(F_2(K))|&=\binom{s+1}{2}-\iota(K).
\end{aligned}
\end{equation}
There are \(s-\iota(K)\) internal sectors and \(\binom s2\) mixed sectors. Since an edge exists in the root graph,
\(0\le\iota(K)\le s-1\), giving
\[
 \binom s2<|\pi_0(F_2(K))|\le\binom{s+1}2.
\]
For positive integers \(s\), these intervals are disjoint and consecutive.
Thus the number of token components determines \(s\), and the equality
above then determines \(\iota(K)\).
In particular \(G,H\) have the same number of isolates, say \(\iota\).

\emph{Step 2: recover components of order at least three.}
For a connected graph \(\Gamma\) of order at least three, write
\(\mu_K(\Gamma)\) for the multiplicity of that component type in
\(K\). We claim
\begin{equation}
 \mu_{F_2(K)}(F_2(\Gamma))
 =\mu_K(\Gamma)+\iota(K)\mu_K(F_2(\Gamma)).
\end{equation}
Since \(F_2(\Gamma)\) has at least three vertices, sectors consisting
of one configuration cannot contribute. By connected uniqueness, each other internal component contributes
precisely when its root is isomorphic to \(\Gamma\). A mixed component with both
factors of order at least two cannot contribute, since
\(F_2(\Gamma)\) is Cartesian-prime by Lemma~\ref{lem:prime}.
A mixed component with one singleton factor is just its other factor;
it contributes precisely when that component of the root graph has type
\(F_2(\Gamma)\), once for each isolate. These possibilities
exhaust the occupancy decomposition and prove the formula.

The token isomorphism equates the left sides for \(G,H\). Put
\(\Delta\mu=\mu_G-\mu_H\). Then for every such \(\Gamma\),
\[
 \Delta\mu(\Gamma)+\iota\,\Delta\mu(F_2(\Gamma))=0.
\]
If a multiplicity difference is nonzero, choose a component type of
largest order with this property. This is possible because only finitely
many component types occur in the two roots. At order \(r\ge4\), the graph \(F_2(\Gamma)\) has
\(\binom r2>r\) vertices. Its multiplicity difference is zero by
maximality, forcing \(\Delta\mu(\Gamma)=0\), a contradiction.
At order three, the omitted-vertex identification in
Lemma~\ref{lem:small-unique} gives \(F_2(\Gamma)\cong\Gamma\).
The equation becomes \((1+\iota)\Delta\mu(\Gamma)=0\), again
impossible for a nonzero integer difference.
Thus all component multiplicities of order at least three agree.

\emph{Step 3: recover the remaining components.}
The isolates already agree. Subtract their number and the total number
of vertices in all larger components from the common root order \(n\).
The remaining vertices occur in components of type \(K_2\), so half
this remainder determines their multiplicity. Every connected component
type now has the same multiplicity in \(G,H\), proving \(G\cong H\).
\end{proof}

\section{Automorphisms fixing \texorpdfstring{$E(G)$}{E(G)} pointwise}

Let \(G\) be a finite connected simple graph of order at least three, and put
\(X=F_2(G)\). Write \(\mathcal K_E(G)\) for the pointwise stabilizer
of \(E(G)\) in \(\Aut(X)\). It is the kernel of the permutation action
of the setwise stabilizer of \(E(G)\) on \(E(G)\). Define the distance-vector map
\[
  D(v)=\bigl(\dist_G(v,t)\bigr)_{t\in V(G)}.
\]
It is injective: the coordinate indexed by \(v\) vanishes only at \(v\).

\begin{lemma}\label{lem:edge-state-data}
Let \(\alpha\in\mathcal K_E(G)\). If \(\alpha(xy)=zw\), then
\begin{equation}\label{eq:edge-state-distance-sum}
  D(x)+D(y)=D(z)+D(w).
\end{equation}
If \(\dist_G(x,y)=2\), then either \(xy\) is fixed or \(\alpha\) exchanges
\(xy\) with \(zw\), where the four vertices are distinct, \(xy,zw\) are
nonedges, and
\begin{equation}\label{eq:edge-state-square-profile}
  N_G(x)\cap N_G(y)=\{z,w\},\qquad
  N_G(z)\cap N_G(w)=\{x,y\}.
\end{equation}
\end{lemma}
\begin{proof}
Every configuration of \(\delta_G(t)\) is fixed by \(\alpha\), so
Lemma~\ref{lem:connectivity} gives \eqref{eq:edge-state-distance-sum}
for each coordinate \(t\).

For a nonedge \(xy\) at distance two, put
\(C=N_G(x)\cap N_G(y)\ne\varnothing\). By
\eqref{eq:marked-neighborhood}, its neighbours in \(E(G)\), viewed
as edges of \(G\), form the connected complete bipartite graph
\(K_{\{x,y\},C}\). Since \(\alpha\) fixes \(E(G)\) pointwise and is a
bijection, it cannot send \(xy\notin E(G)\) into \(E(G)\).
Thus \(\alpha(xy)\notin E(G)\) has the same neighbours in \(E(G)\),
and its endpoints form a two-element part of the same bipartite graph.
A connected bipartite graph has a unique bipartition up to exchange,
since choosing the part of one vertex determines the others along paths.
If the configuration moves, its image must therefore be \(zw\), with
\(C=\{z,w\}\). Equality of the neighbour sets in \(E(G)\) gives
\eqref{eq:edge-state-square-profile}. Exactly these two nonedges have that
neighbour set in \(E(G)\), so \(\alpha\) exchanges them.
\end{proof}

\begin{lemma}\label{lem:edge-state-extension}
Let \(G\) be a finite connected simple graph of order at least three, and let
\(\alpha,\beta\in\Aut(F_2(G))\). If \(\alpha(xy)=\beta(xy)\)
for every configuration \(xy\) with \(\dist_G(x,y)\le2\), then \(\alpha=\beta\).
\end{lemma}

\begin{proof}
Put \(\gamma=\beta^{-1}\alpha\). It fixes every configuration \(xy\) with \(\dist_G(x,y)\le2\).
Induct on \(m=\dist_G(x,y)\ge3\), and choose a shortest path
\(x=x_0,x_1,\ldots,x_m=y\). The disjoint configurations \(x_1x_m\) and
\(x_0x_{m-1}\) are fixed by induction. Their common neighbours are exactly
\[
  \{x_0x_m,\ x_1x_{m-1}\}.
\]
Indeed one cross matching uses \(x_0x_1,x_{m-1}x_m\), while the other would
require the absent edge \(x_0x_m\). The inner pair has distance \(m-2\) and is fixed by induction.
Hence the other common neighbour \(x_0x_m\) is also fixed. Thus \(\gamma=1\), so \(\alpha=\beta\).
\end{proof}

\begin{lemma}\label{lem:edge-state-local-partition}
Let \(G\) be a finite connected simple graph of order at least three, and let
\(\alpha\in\mathcal K_E(G)\). At each vertex \(v\), its incident
edges have a unique unlabelled partition into at most two nonempty classes
such that, for distinct
neighbours \(a,b\) of \(v\), the configuration \(ab\) moves exactly when \(va,vb\)
belong to different classes. Whenever \(va,vb\) lie in different local
classes, there is a unique \(d\) such that \(\alpha\) exchanges \(ab,vd\). The vertices \(v,a,d,b\) form
the unique chordless 4-cycle through \(va,vb\), consecutive edges are
in different local classes at every corner, and
\begin{equation}
  D(v)+D(d)=D(a)+D(b).
\end{equation}
\end{lemma}
\begin{proof}
For distinct neighbours \(a,b\) of \(v\), put \(t_v(a,b)=1\) when
\(\alpha(ab)\ne ab\), and zero otherwise. We first prove the parity
identity
\begin{equation}\label{eq:edge-state-local-parity}
 t_v(a,b)+t_v(a,c)+t_v(b,c)=0\quad\text{in }\mathbb F_2
\end{equation}
for three distinct neighbours \(a,b,c\) of \(v\).
If all three values vanish, the identity holds. Otherwise relabel so
that \(ab\) moves. By Lemma~\ref{lem:edge-state-data}, the configurations
\(ab,vd\) are exchanged and
\[
 N_G(a)\cap N_G(b)=\{v,d\},\qquad
 N_G(v)\cap N_G(d)=\{a,b\}.
\]
Since \(vd\) is a nonedge and \(vc\) an edge, \(d\ne c\).
The four distinct configurations \((ac,cv,bc,cd)\) form a cycle, and
\eqref{eq:degrees} gives both opposite pairs codegree two.
Put \(q=\alpha(cd)\).
Its image is the 4-cycle \((\alpha(ac),cv,\alpha(bc),q)\), since \(cv\) is fixed.
Also \(cd\sim vd\) along the edge \(cv\) of \(G\), so \(q\sim ab\).

A moved \(ac\) or \(bc\) is a nonedge at distance two through \(v\),
so Lemma~\ref{lem:edge-state-data} gives image \(ve\), with \(e\) outside
the original pair and \(v\).
Thus \(\alpha(ac)\) contains exactly one of \(c,v\), namely \(c\) if \(ac\) is fixed;
similarly for \(\alpha(bc)\).
By Lemma~\ref{lem:token-square-types}, the image cycle either shares
a vertex of \(G\) or is a rectangle. A shared vertex must be \(c\) or \(v\).
If it is \(c\), then \(\alpha(ac)=ac\), \(\alpha(bc)=bc\), and \(q\sim ab\) forces
\(q=ac\) or \(bc\), contradicting distinctness of the image configurations.
If it is \(v\), adjacency to \(ab\) forces \(q=va\) or \(vb\).
These belong to \(E(G)\) and are fixed, so injectivity gives \(cd=va\) or \(vb\),
impossible. Hence the image is a rectangle. At its corner \(cv\),
the neighbours \(\alpha(ac),\alpha(bc)\) contain different members of \(\{c,v\}\).
Exactly one of \(ac,bc\) therefore moves, proving
\eqref{eq:edge-state-local-parity}.

Connectedness and order at least three give a neighbour at every vertex.
Choose a reference neighbour \(a_0\) at \(v\). Define
\(\ell_v(a_0)=0\) and \(\ell_v(a)=t_v(a_0,a)\) otherwise.
For distinct neighbours, parity gives
\(t_v(a,b)=\ell_v(a)+\ell_v(b)\); if one is \(a_0\), the same
identity is the definition. The two possible values of \(\ell_v\)
therefore partition the incident edges into at most two classes.
At degrees one and two this construction works directly without a
triple equation. Changing the reference neighbour adds one constant
to all labels, so the unlabelled partition does not depend on the choice.
For distinct neighbours, any partition with the stated property puts
\(va,vb\) in the same class exactly when \(t_v(a,b)=0\). This relation
determines its classes uniquely.

For a mixed pair, \(t_v(a,b)=1\). The profile equalities determine
the other diagonal \(vd\) uniquely. They also give the four cross edges
and exclude both diagonals, hence give a unique chordless 4-cycle.
Both diagonal configurations move, so the two incident edges at every
corner belong to different local classes. Its parallelogram identity is
exactly \eqref{eq:edge-state-distance-sum} for the exchanged configurations.
\end{proof}

\section{Cartesian structure and the pointwise stabilizer}
\begin{definition}\label{def:coordinate-exchange}
For a fixed identification in a Cartesian product decomposition \(G\cong A\square B\), its \emph{coordinate exchange} is the transported involution
\(\xi\bigl(\{(a,b),(a',b')\}\bigr)=\{(a,b'),(a',b)\}\).
\end{definition}
Coordinate permutations of this kind are used in
\cite[Proof of Theorem~1.2]{cartesian-aut}.
\begin{lemma}\label{lem:exchange-sound}
A coordinate exchange is a token automorphism fixing \(E(G)\) pointwise. Conjugation by an automorphism of the root graph gives the exchange in the transported Cartesian product decomposition.
\end{lemma}
\begin{proof}
The token graph is the quotient of
\(G\square G-\{(x,x):x\in V(G)\}\) by interchanging the two
entries: each edge moves one token to an unoccupied vertex.
For \(G=A\square B\), swapping the two \(B\)-coordinates in
\(A\square B\square A\square B\) is a graph automorphism. It
preserves the deleted diagonal and commutes with interchange of the two entries,
so induces \(\xi\) on the quotient. A configuration in \(E(G)\) has only one
differing factor coordinate, hence is fixed. Relabelling the product
decomposition gives the conjugation assertion.
\end{proof}

\begin{lemma}\label{lem:product-twins}
The only connected nontrivial Cartesian product with twin vertices is $C_4$.
\end{lemma}
\begin{proof}
Consider twins in \(A\square B\), where both factors
are connected and nontrivial. Suppose first that twin vertices agree
in their first coordinate, say \((a,b),(a,b')\), with \(b\ne b'\).
Choose a neighbour \(a_1\) of \(a\). Then \((a_1,b)\) is adjacent
to the first vertex but differs from the second in both coordinates,
contradicting twinness. The same argument rules out agreement in the second coordinate.

Thus twins have the form \((a,b),(a',b')\) with both coordinates
different. Every neighbour \((a_1,b)\) of the first vertex must be
adjacent to the second, forcing \(a_1=a'\). Reversing the twins shows
that \(a\) is the only neighbour of \(a'\). Connectedness gives
\(A=K_2\). The same argument gives \(B=K_2\), so the product is
\(C_4\). Its opposite vertices are twins, verifying the converse.
\end{proof}

\begin{lemma}[Affine Cartesian-product criterion]
\label{lem:distance-product}
Let \(\Gamma\) be a finite connected simple graph. At each vertex
partition its incident edges into one or two nonempty unlabelled classes,
with two at some vertex. Let \(\Lambda\) be a torsion-free abelian group
and \(F:V(\Gamma)\to\Lambda\) injective. Suppose that for every pair
of incident edges \(vx,vy\) in different local classes there exists a
4-cycle \(v,x,z,y\) alternating between
the local classes at every corner and satisfying the parallelogram identity
\[
 F(v)+F(z)=F(x)+F(y).
\]
Then both classes occur everywhere and have consistent global names
red and blue. For \(o\in V(\Gamma)\), let \(A,B\) be its red and blue
components, using only edges of the indicated colour.
Then \(\Gamma\cong A\square B\) with these prescribed factor directions,
and both factors are connected and nontrivial. In particular the two
global edge classes have the unique square property: every mixed corner
spans a unique chordless 4-cycle with opposite edges in the same class.
\end{lemma}
\begin{proof}
A fourth vertex satisfying the parallelogram identity is unique: its value
is \(F(x)+F(y)-F(v)\), and \(F\) is injective. We use only 4-cycles
satisfying this identity.

\emph{Step 1: give the local classes consistent global names.}
A two-class vertex cannot have a one-class neighbour: an alternating
square through the joining edge would require two classes there.
Connectedness gives two classes everywhere. Choose names at \(o\)
and carry them along a walk, agreeing on each traversed edge.
In a 4-cycle satisfying the parallelogram identity \(v,x,z,y\),
\[
 F(z)-F(x)=F(y)-F(v).
\]
Thus transporting an edge across an edge of the other class preserves
its increment. Replacing a blue-red turn by the other two sides of
its alternating square changes the turn to red-blue. At the opposite
corner the original and replacement paths arrive along the square's
red and blue edges, which belong to different local classes. Both paths
therefore assign the same names to those classes. The replacement
preserves the terminal naming, including the names along any remaining
suffix. Each replacement decreases the number of blue-before-red pairs.
Repeating it puts all red steps before all blue steps without changing
the terminal naming.

Suppose a closed walk reversed the names, and sort it as a red walk
\(o\leadsto p\) followed by a blue walk \(p\leadsto o\).
Choose an initially blue edge \(ob\), and put \(\lambda=F(b)-F(o)\).
Transport it along the red part to obtain a blue edge at \(p\) with
increment \(\lambda\). In the terminal naming at \(o\), the same edge
\(ob\) is red. Transporting it backwards along the blue part gives
a red edge at \(p\) with increment \(\lambda\). Reversing the blue part
also reverses its name transport, recovering at \(p\) exactly the naming
obtained after the red part. Injectivity makes the other endpoints equal, so one edge would belong
to both local classes, a contradiction. Hence every closed walk preserves
the names, and the naming is independent of the path.

\emph{Step 2: fill finite rectangles.}
A red walk \(s=a_0,\ldots,a_r=a\) and a blue walk
\(s=b_0,\ldots,b_t=b\) form adjacent boundaries of a rectangle.
Fill its cells successively by the alternating 4-cycles supplied by the hypothesis. Opposite
boundaries are red and blue walks to one endpoint \(z\), and induction gives
\[
 F(z)=F(a)+F(b)-F(s).
\]
At every cell the two boundary edges lie in different local classes
and are therefore distinct, so the 4-cycle-completion hypothesis applies.
Empty boundary walks are allowed; distant cells may repeat vertices.
If a red and a blue component met at distinct \(x,y\), start with
\(z_0=x,z_1=y\) and walks of both colours from \(z_0\) to \(z_1\).
Fill their rectangle. Its opposite boundaries are walks of both colours
from \(z_1\) to the new endpoint \(z_2\). Repeating this construction
gives walks of both colours from \(z_j\) to \(z_{j+1}\), with
\[
 F(z_{j+1})=2F(z_j)-F(z_{j-1}),\qquad j\ge1.
\]
Induction therefore gives
\[
 F(z_j)=F(x)+j(F(y)-F(x)),\qquad j=0,1,\ldots.
\]
The increment is nonzero by injectivity. Since the group is torsion-free,
these values are distinct, contradicting finiteness. Thus a red and a blue component
intersect in at most one vertex.

\emph{Step 3: the rectangles give the product.}
For \(a\in V(A),b\in V(B)\), define \(p(a,b)\) as the endpoint of
their rectangle from \(o\). Its value is \(F(a)+F(b)-F(o)\), so injectivity
makes the endpoint independent of the chosen walks. It lies in the blue
component of \(a\) and the red component of \(b\).
If two pairs have the same image, their \(A\)-vertices lie in the same
blue component. They also lie in the red component \(A\), so the preceding
intersection result makes them equal. The \(B\)-vertices are equal by
the same argument. Thus \(p\) is injective.

Appending a red edge \(aa'\) to a boundary walk and filling the
same blue boundary gives the red edge \(p(a,b)p(a',b)\).
The same construction with the colours interchanged preserves blue
adjacency. Conversely, take a red edge
\(p(a,b)z'\). Its rectangle supplies a blue walk from \(a\) to \(p(a,b)\).
Transport this red edge backwards along that walk, one square at a time.
It ends at a red neighbour \(a'\in V(A)\) of \(a\), preserving the increment:
\[
 F(z')-F(p(a,b))=F(a')-F(a).
\]
The formula for \(F(p(a',b))\) and injectivity give \(z'=p(a',b)\).
The same argument with the colours interchanged applies to blue edges.
Thus \(p\) preserves and reflects edges, and its image contains every
neighbour of each of its vertices. Since it contains \(o=p(o,o)\),
connectedness implies that it is all of \(V(\Gamma)\).
Therefore \(p:A\square B\to\Gamma\) is the required isomorphism.
Both factors are nontrivial because both colours occur at \(o\).
In a product a mixed corner has its unique coordinate fourth vertex
and neither diagonal is an edge, proving the chordless-square assertion.
\end{proof}

\begin{theorem}[The pointwise stabilizer of \(E(G)\)]\label{thm:edge-state-kernel}
Let \(G\) be a finite connected simple graph of order at least three, with
Cartesian prime factorization \(G=G_1\sq\cdots\sq G_r\), where each
factor has at least two vertices. Write \(\xi_J\) for the exchange of
the two tokens' coordinates indexed by \(J\subseteq\{1,\ldots,r\}\).
Then
\[
 \mathcal K_E(G)=\{\xi_J:J\subseteq\{1,\ldots,r\}\}
 \cong C_2^r/\langle(1,\ldots,1)\rangle\cong C_2^{\,r-1}.
\]
For subsets \(J,K\subseteq\{1,\ldots,r\}\),
\[
 \xi_J=\xi_K\quad\Longleftrightarrow\quad K=J\text{ or }K=J^c,
\]
where \(J^c=\{1,\ldots,r\}\setminus J\).
In particular \(\mathcal K_E(G)=1\) if and only if \(G\) is Cartesian-prime.
\end{theorem}
\begin{proof}
Take \(1\ne\alpha\in\mathcal K_E(G)\).
By Lemma~\ref{lem:edge-state-local-partition}, the incident edges have
local partitions, and the injective map \(D\) satisfies the parallelogram
identity on their alternating 4-cycles.
By Lemma~\ref{lem:edge-state-extension}, \(\alpha\ne1\) moves a configuration
\(xy\) with \(\dist_G(x,y)\le2\). Since \(E(G)\) is fixed pointwise,
\(\dist_G(x,y)=2\). At a common neighbour the two incident edges lie
in different local classes, so some vertex has two classes.
Lemma~\ref{lem:distance-product}, with
\(\Gamma=G\), \(F=D\) and \(\Lambda=\mathbb Z^{V(G)}\), gives
\(G=A\square B\) with nontrivial connected factors whose directions
agree with the local classes. Thus a Cartesian-prime root has trivial kernel.

Let \(\xi\) exchange the two tokens' \(B\)-coordinates.
For a configuration \(xy\) with \(\dist_G(x,y)=2\), choose a common neighbour \(v\).
If \(vx,vy\) have the same factor direction, \(x,y\) agree in the
other coordinate, so \(\xi\) fixes their unordered pair, as does
\(\alpha\) by Lemma~\ref{lem:edge-state-local-partition}.
If the directions differ, \(xy\) is a diagonal of the unique coordinate
rectangle, and both maps exchange its diagonals by that lemma.
By Lemma~\ref{lem:exchange-sound}, \(\xi\) is an automorphism fixing
\(E(G)\) pointwise. Thus \(\alpha,\xi\) agree on every configuration \(xy\)
with \(\dist_G(x,y)\le2\), and Lemma~\ref{lem:edge-state-extension} gives \(\alpha=\xi\).

Refine \(A,B\) into Cartesian primes. The coordinate form of prime
uniqueness \cite[Theorems~6.6 and~6.8]{handbook} identifies their directions
with two unions of the fixed prime directions, up to a permutation and
separate factor isomorphisms. Applying a factor isomorphism to both tokens
commutes with swapping their coordinates, so \(\xi=\xi_J\) for
some \(J\). The identity is \(\xi_\varnothing\), so every kernel
element has this form. The same coordinate description gives, for an
automorphism \(\sigma\) of \(G\), let \(\pi_\sigma\) send each
prime-factor direction to its image direction under \(\sigma\). Then
\begin{equation}\label{eq:prime-exchange-action}
 \sigma^{(2)}\xi_J(\sigma^{(2)})^{-1}
 =\xi_{\pi_\sigma(J)}.
\end{equation}

By Lemma~\ref{lem:exchange-sound}, every \(\xi_J\) belongs to the
kernel, including the identity exchanges for the empty and full subsets.
Swapping coordinates twice cancels, so
\(\xi_J\xi_K=\xi_{J\triangle K}\).
Thus \(\mathbf1_J\mapsto\xi_J\) defines a surjective homomorphism
\(C_2^r\to\mathcal K_E(G)\), where \(\mathbf1_J\) is the characteristic
vector of \(J\). The all-ones vector exchanges the two whole tokens,
so acts trivially on unordered pairs.
For a nonempty proper \(J\), choose \(i\in J,j\notin J\) and two
vertices differing in exactly coordinates \(i,j\), possible since both
factors have at least two vertices. Their pair is changed to the other
diagonal by \(\xi_J\), so \(\xi_J\ne1\).
The homomorphism therefore has kernel \(\langle(1,\ldots,1)\rangle\).
Taking its quotient gives the stated group and the equivalence
\(\xi_J=\xi_K\) exactly when \(J\triangle K\) is empty or full.
Since \(r\ge1\), the group is trivial exactly when \(r=1\), completing
the Cartesian-prime characterization.
\end{proof}

\section{Switch subgroup and induced automorphisms}
We now determine the subgroup generated by component switches and its
intersection with the induced automorphisms. The resulting quotient also
describes the choice of correction in the normalization proof.
For connected \(G\) of order at least five, use \(I_G,\mathcal T_G\) from the
Introduction. For \(\varepsilon\in\mathbb F_2^{I_G}\), set
\[
 \tau_\varepsilon=\prod_{(S,C)\in I_G}\tau_{S,C}^{\varepsilon_{S,C}}.
\]
Here \(\mathbf1_S\) is one on the coordinates indexed by \(S\) and zero
elsewhere; \(\mathbf1\) denotes the constant one-vector within a block.
Write \(s_S\) for the transposition of the endpoints of
\(S\). Empty products are the identity and empty sums are the zero space.

\begin{proposition}\label{prop:switch-module}
Let \(G\) be connected of order at least five. The map
\(\tau:\mathbb F_2^{I_G}\to\Aut(F_2(G))\) is an injective homomorphism;
write \(\mathcal S_G=\tau(\mathbb F_2^{I_G})\). Then \(\mathcal T_G\) is normal in
\(\Aut(G)\), \(|\mathcal T_G|=2^{|\mathcal Q_2(G)|}\), and
\[
\begin{aligned}
 \tau_{\mathbf1_S}&=s_S^{(2)},\\
 \tau\Bigl(\bigoplus_{S\in\mathcal Q_2(G)}\langle\mathbf1_S\rangle\Bigr)
   &=\mathcal T_G^{(2)},\\
 \mathcal S_G\cap\Aut(G)^{(2)}&=\mathcal T_G^{(2)}.
\end{aligned}
\]
The coordinate action
\[
 (h\cdot\varepsilon)_{S,C}
   =\varepsilon_{h^{-1}(S),h^{-1}(C)}
\]
satisfies \(h^{(2)}\tau_\varepsilon(h^{(2)})^{-1}=\tau_{h\cdot\varepsilon}\).
It preserves \(\bigoplus_{S\in\mathcal Q_2(G)}\langle\mathbf1_S\rangle\), and \(\mathcal T_G\) acts trivially on \(I_G\), so
the action factors through \(\Aut(G)/\mathcal T_G\). Moreover,
\[
\begin{aligned}
 \mathbb F_2^{I_G}\Big/\Bigl(
   \bigoplus_{S\in\mathcal Q_2(G)}\langle\mathbf1_S\rangle\Bigr)
 &\cong\bigoplus_{S\in\mathcal Q_2(G)}
       \mathbb F_2^{\pi_0(G-S)}/\langle\mathbf1\rangle\\
 &\cong\mathcal S_G/\mathcal T_G^{(2)}.
\end{aligned}
\]
These quotients have dimension
\(\sum_{S\in\mathcal Q_2(G)}(|\pi_0(G-S)|-1)\).
\end{proposition}
\begin{proof}
By Corollary~\ref{cor:component-switches} and
Lemma~\ref{lem:cut-compatibility}, the switches are commuting involutions.
Thus \(\tau\) is a homomorphism.
For injectivity, fix \(S=\{a,b\}\), \(C\in\pi_0(G-S)\) and \(x\in V(C)\).
Apply the switches based at \(S\) last. Every other cut is disjoint
from \(S\), so its switches fix \(a\). Its endpoints lie in one component
of \(G-S\), so they also leave the other token in \(C\).
Of the switches based at \(S\), only \(\tau_{S,C}\) can then act. Thus
\(\tau_\varepsilon(ax)\) contains \(b\) rather than \(a\) exactly when
\(\varepsilon_{S,C}=1\). If the automorphism is the identity, every coordinate must vanish.
This proves injectivity.

The product over all components of \(G-S\) exchanges exactly the configurations
containing one endpoint of \(S\), so \(\tau_{\mathbf1_S}=s_S^{(2)}\).
The transpositions have disjoint supports and are independent, giving
the stated order of \(\mathcal T_G\). Conjugation sends \(s_S\) to
\(s_{h(S)}\), so this subgroup is normal.
Formula~\eqref{eq:switch-conjugation} gives the coordinate action and
\(h\mathbf1_S=\mathbf1_{h(S)}\). A generator \(s_R\) fixes \(G-R\) pointwise. For another cut \(S\),
the endpoints of \(R\) lie in one component of \(G-S\).
Thus \(s_R\) fixes \(S\) and every component of \(G-S\) setwise,
and \(\mathcal T_G\) fixes every index.

Suppose \(\tau_\varepsilon=h^{(2)}\). By the switch definition and cut disjointness, every switch fixes each
cut as a token configuration. Hence \(h(S)=S\).
For \(S=\{a,b\}\) and \(x\in V(C)\), the image \(h^{(2)}(ax)\) contains
\(b\) exactly when \(h(a)=b\), since \(h(x)\notin S\).
The preceding criterion therefore makes \(\varepsilon_{S,C}\) independent
of \(C\). Thus \(\varepsilon\in\bigoplus_{S\in\mathcal Q_2(G)}\langle\mathbf1_S\rangle\).
The reverse inclusion follows from
\(\tau(\bigoplus_{S\in\mathcal Q_2(G)}\langle\mathbf1_S\rangle)=\mathcal T_G^{(2)}\),
proving the intersection equality.

Injectivity and the equality for the constant subspaces give the quotient
isomorphism. Decomposing by cuts gives the direct sum.
Each block has dimension \(|\pi_0(G-S)|\), and its constant subspace has
dimension one. Subtracting these dimensions and summing proves the formula.
\end{proof}

In the normalization argument, put $\varepsilon_{S,C}=\eps_S(C)$ at
the selected cuts and zero elsewhere. The correction is then
$\tau_\varepsilon$. Reversing the order of the source twin pair or its corresponding target
pair at $S$ adds $\mathbf1_S$ to $\varepsilon$ and changes the correction
by $s_S^{(2)}$. Thus the ambiguity independent of these orderings is
$\mathcal S_G/\mathcal T_G^{(2)}$, the quotient by induced twin
transpositions described in Proposition~\ref{prop:switch-module}.

\begin{corollary}\label{cor:ordered-cuts}
Choose \(O(S)=(a_S,b_S)\) for each twin 2-cut, and put
\[
 \Aut(G;O)=\{h\in\Aut(G):h(a_S)=a_{h(S)},\ h(b_S)=b_{h(S)}
                    \text{ for every }S\}.
\]
Under the hypotheses of Proposition~\ref{prop:switch-module},
\(\Aut(G)=\mathcal T_G\rtimes\Aut(G;O)\), \(\Aut(G;O)\cong\Aut(G)/\mathcal T_G\), and
\[
 \mathcal S_G\Aut(G;O)^{(2)}
 \cong\mathbb F_2^{I_G}\rtimes(\Aut(G)/\mathcal T_G)
\]
with the coordinate action in that proposition.
\end{corollary}
\begin{proof}
Composition and inversion preserve the ordered-pair condition, so
\(\Aut(G;O)\) is a subgroup. For \(h\in\Aut(G)\), transpose the endpoints of each image cut on which
\(h\) reverses the chosen ordering. This gives \(dh\in\Aut(G;O)\), with
\(d\in \mathcal T_G\). The correction is unique because distinct cuts
are disjoint.
Consequently \(\mathcal T_G\cap\Aut(G;O)=1\), and the quotient restricts to the
stated isomorphism. The pair action is faithful by
\eqref{eq:pair-map-faithfulness}.
Proposition~\ref{prop:switch-module} therefore gives
\(\mathcal S_G\cap\Aut(G;O)^{(2)}=1\).
Its conjugation formula shows that \(\Aut(G;O)^{(2)}\) normalizes
\(\mathcal S_G\), proving the semidirect product. In quotient coordinates
its multiplication is
\[
 (\varepsilon,\overline h)(\eta,\overline k)
   =(\varepsilon+\overline h\cdot\eta,\overline{hk}).
\]
\end{proof}

\section{Automorphism groups and isomorphism normal form}
\begin{proof}[Proof of Theorem~\ref{thm:groups}]
Lemma~\ref{lem:product-twins} shows that the cases are disjoint.
By Proposition~\ref{prop:switch-module}, $\mathcal T_G$ is normal.
The induced pair action is faithful by \eqref{eq:pair-map-faithfulness}.

Let $\alpha\in\Aut(F_2(G))$. By Theorem~\ref{thm:edge-state-preservation},
$\alpha(E(G))=E(G)$. Theorem~\ref{thm:marked-map-normalization} gives
$\theta\in\mathcal S_G$ and $\sigma\in\Aut(G)$ with
$\alpha\theta^{-1}(e)=\sigma^{(2)}(e)$ for every $e\in E(G)$. Hence
\[
 \beta=(\sigma^{(2)})^{-1}\alpha\theta^{-1}\in\mathcal K_E(G),
 \qquad \alpha=\sigma^{(2)}\beta\theta.
\]
Conversely, every factor on the right consists of token automorphisms, so
\begin{equation}\label{eq:automorphism-factorization}
 \Aut(F_2(G))=\Aut(G)^{(2)}\mathcal K_E(G)\mathcal S_G.
\end{equation}
The right-hand side denotes the set of products in this order.
Since every token automorphism preserves $E(G)$, the subgroup
$\mathcal K_E(G)$ is the kernel of the action on $E(G)$ and is normal.

We will use
\begin{equation}\label{eq:kernel-induced-intersection}
 \mathcal K_E(G)\cap\Aut(G)^{(2)}=1.
\end{equation}
Indeed, if $\rho^{(2)}\in\mathcal K_E(G)$, then $\rho$ fixes each edge
of $G$ setwise. Since the graph is connected and has at least three vertices, two
distinct edges meet at a vertex. Their intersection is fixed.
For an edge fixed setwise, if one endpoint is fixed, then so is the other.
Applying this along paths gives $\rho=1$, proving the intersection.

\emph{The cut case.}
If $\mathcal Q_2(G)\ne\varnothing$, Lemma~\ref{lem:product-twins} makes
$G$ Cartesian-prime, so Theorem~\ref{thm:edge-state-kernel} gives
$\mathcal K_E(G)=1$. Equation~\eqref{eq:automorphism-factorization}
reduces to $\Aut(G)^{(2)}\mathcal S_G$. The conjugation identity
\eqref{eq:switch-conjugation} shows that the induced subgroup normalizes
$\mathcal S_G$, so this product equals $\mathcal S_G\Aut(G)^{(2)}$.
Choose orderings as in Corollary~\ref{cor:ordered-cuts}. Since
$\mathcal T_G^{(2)}\subseteq\mathcal S_G$, that corollary gives
\begin{equation}\label{eq:cut-aut-decomposition}
 \Aut(F_2(G))=\mathcal S_G\rtimes\Aut(G;O)^{(2)}.
\end{equation}
Corollary~\ref{cor:ordered-cuts} gives the stated abstract group and action.
Proposition~\ref{prop:switch-module} also gives
\[
 |\Aut(F_2(G))|=2^{|I_G|-|\mathcal Q_2(G)|}|\Aut(G)|
 =2^{\sum_{S\in\mathcal Q_2(G)}(|\pi_0(G-S)|-1)}|\Aut(G)|.
\]

\emph{The product case.}
If $G$ is a nontrivial Cartesian product, Lemma~\ref{lem:product-twins}
excludes twins at order at least five. Thus $\mathcal S_G=1$.
Equations~\eqref{eq:automorphism-factorization}
and~\eqref{eq:kernel-induced-intersection}, with normality of the kernel, give
\begin{equation}\label{eq:product-aut-decomposition}
 \Aut(F_2(G))=\mathcal K_E(G)\rtimes\Aut(G)^{(2)}.
\end{equation}
Theorem~\ref{thm:edge-state-kernel} identifies the first subgroup with
$C_2^r/\langle(1,\ldots,1)\rangle\cong C_2^{r-1}$.
By \eqref{eq:prime-exchange-action}, conjugation permutes the prime
coordinates. This preserves the all-ones subgroup and defines the stated
action on the quotient. The product formula follows, with order
$2^{r-1}|\Aut(G)|$.

\emph{The remaining case.}
Here $G$ is Cartesian-prime and has no twin 2-cut, so
$\mathcal K_E(G)=\mathcal S_G=1$. Equation~\eqref{eq:automorphism-factorization}
gives
\begin{equation}\label{eq:rigid-aut-decomposition}
 \Aut(F_2(G))=\Aut(G)^{(2)}\cong\Aut(G).
\end{equation}
\end{proof}

\begin{proof}[Proof of Theorem~\ref{thm:exhaustion}]
Theorem~\ref{thm:all} gives an isomorphism $\sigma_0:G\to H$.
Then $\alpha=\phi(\sigma_0^{(2)})^{-1}$ belongs to $\Aut(F_2(H))$.
The subgroup decompositions
\eqref{eq:cut-aut-decomposition}, \eqref{eq:product-aut-decomposition}
and~\eqref{eq:rigid-aut-decomposition},
applied to $H$, give $\alpha=\eta\rho^{(2)}$ with $\rho\in\Aut(H)$.
Here $\eta$ is a component-switch product in the cut case, a coordinate
exchange in the product case, and the identity in the remaining case.
Consequently
\[
 \phi=\eta\rho^{(2)}\sigma_0^{(2)}
      =\eta(\rho\circ\sigma_0)^{(2)}.
\]
Put $\sigma=\rho\circ\sigma_0:G\to H$. An isomorphism between $G,H$ preserves Cartesian decompositions and twin
2-cuts, so $\eta$ has the type required in the statement. In the prime case with no twin 2-cut,
$\eta=1$, and \eqref{eq:pair-map-faithfulness} gives uniqueness of $\sigma$.
Theorem~\ref{thm:edge-state-kernel} identifies the coordinate-exchange
subgroup with $\mathcal K_E$; applying
\eqref{eq:kernel-induced-intersection} to each root proves the intersection
assertion.
\end{proof}

For \(G=C_4\), the four edge
configurations and two diagonal configurations form \(F_2(C_4)=K_{4,2}\). The vertices in the two
parts have different degrees, so each part is preserved. Arbitrary
permutations within the parts give all automorphisms, of order \(4!2!=48\). The product-case
formula would give \(2|\Aut(C_4)|=16\), so it does not extend to
order four.
For cyclic vertices \(a,b,c,d\), the component switches
\[
 \tau_{\{a,c\},\{b\}}=(ab\ bc),\qquad
 \tau_{\{b,d\},\{a\}}=(ab\ ad)
\]
are noncommuting transpositions of configurations in \(E(C_4)\).
Each cut separates the endpoints of the other, so the component
conditions used in the commutation argument are not preserved.

\end{document}